\documentclass[11pt,reqno,a4paper]{amsart}
\usepackage{amssymb}
\usepackage{amsmath}
\usepackage[mathscr]{euscript}
\usepackage{hyperref}
\usepackage{graphicx}
\usepackage[normalem]{ulem} 
\usepackage{chngpage}
\usepackage{amssymb}
\usepackage{hyperref}
\RequirePackage[dvipsnames]{xcolor} 
\definecolor{halfgray}{gray}{0.55}
\definecolor{webgreen}{rgb}{0,0.5,0}
\definecolor{webbrown}{rgb}{.6,0,0} 
\hypersetup{%
  colorlinks=true, linktocpage=true, pdfstartpage=3,
  pdfstartview=FitV,%
  breaklinks=true, pdfpagemode=UseNone, pageanchor=true,
  pdfpagemode=UseOutlines,%
  plainpages=false, bookmarksnumbered, bookmarksopen=true,
  bookmarksopenlevel=1,%
  hypertexnames=true,
  pdfhighlight=/O,
  urlcolor=webbrown, linkcolor=RoyalBlue,
  citecolor=webgreen, 
  pdftitle={},%
  pdfauthor={},%
  pdfsubject={2000 MAthematical Subject Classification: Primary:},%
  pdfkeywords={},%
  pdfcreator={pdfLaTeX},%
  pdfproducer={LaTeX with hyperref}%
}

\usepackage{enumitem}

\newtheorem{theorem}{Theorem}

\newtheorem{lemma}{Lemma}

\numberwithin{equation}{section}

\renewcommand{\epsilon}{\varepsilon}

\allowdisplaybreaks[1]

\begin{document}

\title{HIGHER REGULARITY OF SOLUTIONS OF AN
ITERATIVE FUNCTIONAL EQUATION}

\author{Liang Feng}
\address[Liang Feng]{School of Mathematical Sciences, Chongqing Normal University, Chongqing 401331, China}
\email{fengliang199710@163.com}

\author{Xiao Tang}
\address[Xiao Tang]{School of Mathematical Sciences, Chongqing Normal University, Chongqing 401331, China}
\email[Corresponding author]{mathtx@163.com}



\maketitle

 \begin{abstract}
In this paper, we investigate the existence of $C^n$, $n\in \mathbb{N}^+$, solutions  for a class of second-order iterative functional equations involving iterates of the unknown function and a nonlinear term. Applying the Fiber Contraction Theorem and Fa\`a di Bruno's Formula, we establish the existence of bounded $C^n$ solutions with bounded derivatives of order from $1$ to $n$.
\end{abstract}

\section{Introduction}
Iterative functional equations, which involve iterations of unknown functions (\cite{Baron2,KM} and references therein), have been widely investigated. 
This class includes iterative roots (\cite{LL1,LL2} and references therein) and polynomial-like iterative equations (\cite{MA,ZW} and references therein). The fundamental challenge stems from  nonlinearity of the iteration operator.

In 2000, N. Brillou\"{e}t-Belluot (\cite{BB1}) first introduced the second-order iterative functional equation
\begin{equation}\label{o2}
    \phi^2(x) = \phi(x+a)-x,
\end{equation}
motivated by the multi-variable functional equation
\begin{equation}\label{o1}
 x+\phi(y+\phi(x)) = y+\phi(x+\phi(y)).
\end{equation}
It is easy to see that \eqref{o1} with $y = 0$ is reduced to \eqref{o2} with $a=\phi(0)$. 
Since the existence of solutions to functional equations must be re‑examined even under small changes, many researchers  studied both equations \eqref{o2} and \eqref{o1} extensively; see \cite{Balcerowski, BB2, JJ, SM, TX, ZY}. By \cite[Corollary 3.8]{DS},  \cite[Theorem 11]{MJ} or \cite[Theorem 5]{YD}, equation \eqref{o2} has no continuous solutions on $\mathbb{R}$ when $a = 0$.
 
In 2010, N. Brillou\"{e}t-Belluot and W. Zhang (\cite{BB2}) considered a general version of equation \eqref{o2},
\begin{equation}\label{o3}
    \phi^2(x) = \lambda\phi(x+a) + \mu x, 
\end{equation}
where $\lambda, a$ and $\mu$ are real constants such that $a\lambda \neq 0$. 
They investigated the existence of Lipschitz solutions of equation \eqref{o3} on any compact interval and gave a construction of piecewise continuous solutions of equation \eqref{o3} on a bounded interval. Later, Y. Zeng and W. Zhang (\cite{ZY}) proved that equation \eqref{o3} has no continuous solutions on $\mathbb{R}$ when $\lambda=1$ and $\mu \le -1$. They also provided a sufficient condition for the existence of continuous solutions of equation \eqref{o3} on $\mathbb{R}$.

In 2018, X. Tang and W. Zhang (\cite{TX}) generalized equation \eqref{o3} to the more general form:
\begin{equation}\label{o4}
    \phi^2(x) = h(\phi(f(x))) +g(x), \qquad x\in \mathbb{R},
\end{equation}
where $h, f$ and $g$ are given functions, and $\phi$ is an unknown one. 
They proved the existence of bounded Lipschitz solutions to equation \eqref{o4} on $\mathbb{R}$ under Lipschitz condition when $g$ is bounded while the existence of unbounded Lipschitz solutions to the same equation on  $\mathbb{R}$ when $g$ is unbounded under additional bounded nonlinearity conditions.
Furthermore, without Lipschitz condition, they constructed continuous solutions on  $\mathbb{R}$.
In 2024, W. Shi and X. Tang (\cite{tx}) further studied the $C^1$ solutions of equation \eqref{o4} using the Fiber Contraction Theorem.

In this paper, we continue to investigate  $C^n$, $n\ge 2$ being an integer, solutions of equation \eqref{o4}. Imposing constraints on  derivatives  of the given functions and applying the Fiber Contraction Theorem (\cite{VA}) presented in Section \ref{lemma}, we establish the existence of $C^n$ solutions with bounded  derivatives. The concrete result is following:
\begin{theorem}\label{th}
Assume that functions $h:\mathbb{R} \to \mathbb{R}, f:\mathbb{R} \to  \mathbb{R}$ and $g:\mathbb{R} \to \mathbb{R}$ are of class $C^{n}$ such that
\begin{equation}\label{o}
\inf_{x\in \mathbb{R}}|h'(x)|\ge m, 
~\inf_{x\in \mathbb{R}}|f'(x)|\ge \alpha,
\end{equation}
\begin{equation}\label{gjie}
\max\Big\{\sup_{x\in \mathbb{R}}|g'(x)|,\sup_{x\in \mathbb{R}}|g''(x)|,...,\sup_{x\in \mathbb{R}}|g^{(n)}(x)| \Big\}~\le \beta, ~\sup_{x\in \mathbb{R}}|g(x)|<+\infty,~
\end{equation}
 where $m > 1$, $\alpha > 0$ are real constants, and $\beta > 0$ is a sufficiently small real constant.
Furthermore, 
\begin{equation}\label{oo}
\max_{1\le k\le n}\Big\{ \sup_{x\in \mathbb{R}}|h^{(k)}(x)|,
  \sup_{x\in \mathbb{R}}|f^{(k)}(x)|
 \Big\}< +\infty.
 \end{equation}
Then the functional equation \eqref{o4}
has a bounded $C^{n}$ solution whose the first up to $n$-th derivatives are bounded.
\end{theorem}

{\bf Remark:}\label{remark} In Theorem \eqref{th}, the $\beta>0$ is required to be sufficiently small in order to guarantee that in the following there exists an $L$ solving \eqref{3} and a $\rho_1$ solving \eqref{44} and \eqref{55}, ensuring the Fiber Contraction Theorem is applied.

In section \ref{lemma}, we give some preliminaries used in the proof of Theorem \ref{th}. Section \ref{proof} is devoted to proving Theorem \ref{th}.

We give an example to demonstrate our Theorem \ref{th}.
Consider the equation 
\begin{equation}\label{example equation}
    \phi^2(x)=2\phi(3x)+\frac{\text{sin}x}{100},
\end{equation}
which is of the form \eqref{o4} with $h(x)=2x, f(x)=3x$ and $g(x)=\frac{\text{sin}x}{100}$. We can check that $h$, $f$ and $g$ are of class $C^n$ and satisfy conditions \eqref{o}, \eqref{gjie} and \eqref{oo} with constants $m=\frac{3}{2}, \alpha=2$ and $\beta=\frac{1}{100}$. When $\beta=\frac{1}{100}$, $L$ in \eqref{3} exists and we can take $\rho_1=\frac{1}{2}$ such that \eqref{44} and \eqref{55} hold simultaneously. Consequently, by Theorem \ref{th}, equation \eqref{example equation} has a
solution of class $C^n$ and its first to $n$-th derivatives are bounded.

\section{Preliminaries}\label{lemma}
Firstly, we present the Fiber Contraction Theorem, which is the technique we would like to use in the section \ref{proof}.

\begin{lemma} ( \!\cite[Fiber Contraction Theorem ]{VA}) \label{fct}
Let $(Y_{i},d_i), 1\le i\le n,$ be complete metric spaces. Let $\Gamma:Y_{0} \times Y_1\times Y_2 \times \cdot\cdot \cdot \times Y_n\rightarrow Y_{0} \times Y_1\times Y_2 \times \cdot\cdot \cdot \times Y_n$ be a mapping of the form 
$$\Gamma(y_0,y_1,y_2,...,y_n)=(\Phi_0(y_0),\Phi_1(y_0,y_1),\Phi_2(y_0,y_1,y_2),...,\Phi_n(y_0,y_1,y_2,...,y_n)), $$ and such that each $\Phi_k: Y_0\times Y_1 \times\cdot\cdot \cdot\times Y_k  \rightarrow Y_k(0\le k \le n)$ is uniformly  contractive for the first $k$ variables, that is
$$d_{k}(\Phi_k(y_0,y_1,y_2,...,y_k),\Phi_k(y_0,y_1,y_2,...,\tilde{y_k})) \le \gamma  d_{k}(y_k,\tilde{y_k})$$ 
for any $y_i \in Y_i, 0\le i\le k-1, y_k,\tilde y_k \in Y_k$, 
where $0<\gamma <1$ is constant. Then $\Gamma$ has a unique fixed point $(y_{0\infty},y_{1\infty},y_{2\infty},...,y_{n\infty}) \in Y_0\times Y_1 \times\cdot\cdot \cdot\times Y_n. $ If moreover each of the mappings $\Phi_k(\cdot, y_{k\infty}):  Y_0\times Y_1 \times\cdot\cdot \cdot\times Y_{k-1} \rightarrow Y_k(1\le k \le n)$ is continuous, then this fixed point is globally attractive.
\end{lemma}
Next, we state Fa\`a di Bruno’s formula (\cite{FF}), which gives the complete expression for higher‑order derivatives of composite functions, for use in the subsequent proof.
\begin{lemma}( \!\cite[Fa\`a di Bruno's Formula]{FF})\label{FDB}
Let $h,g: \mathbb{R}\to \mathbb{R}$ be sufficiently smooth. Then
    $$(h\circ g)^{(n)}=\sum_{\substack{\ell_1+2\ell_2+\dots+n\ell_n=n \\ \ell_1,\ell_2,\dots,\ell_n\ge 0}}\frac{n!}{\ell_1!\ell_2!\dots \ell_n!} h^{(\ell_1+\ell_2+\dots +\ell_n)}\circ g \cdot \prod_{i=1}^n\bigg(\frac{g^{(i)}}{i!}\bigg)^{\ell_i} ,$$
where $h^{(i)}$ denotes the $i$-th derivative of $h$ for positive integer $i$.
\end{lemma}

We introduce some notations we will use.  
Let $$C_{b}^{0}(\mathbb{R}):= \{ \phi: \mathbb{R} \to \mathbb{R} \mid \phi  \text{ is continuous and} \sup _{x \in \mathbb{R}}|\phi(x)| < +\infty \}, $$
which is a Banach space with the supremum norm:
\[ \|\phi\|:=\sup_{x\in\mathbb{R}} |\phi(x)| \qquad \text{for every } \phi \in C_{b}^{0}(\mathbb{R}).\]
If a function $\phi$ satisfies 
\[ |\phi(x)-\phi(y)|\le L|x-y| \qquad \text{for all } x,y\in\mathbb R\]
and a constant $L>0$, we say that $\phi$ is Lipschitz on $\mathbb{R}$ with 
\[ \text{Lip}(\phi) := \sup_{x,y\in \mathbb R, x\ne y} \frac{|\phi(x)-\phi(y)|}{|x-y|} \le L.\]
For a constant $L > 0 $, let
\begin{equation*}
C_{b}^{0,1}(\mathbb{R};L) := C_{b}^{0}(\mathbb{R}) \cap \{ \phi: \mathbb{R} \to \mathbb{R} \mid \phi ~\text{is Lipschitz $\mathbb{R}$   such that } \text{Lip}(\phi) \le L \},
\end{equation*}
and for a constant $ \rho > 0 $, let
$$ C_{b,\rho}^0(\mathbb{R}) := C_{b}^{0}(\mathbb{R}) \cap \{ \phi: \mathbb{R} \to \mathbb{R} \mid \|\phi\| \le \rho \}.$$
Let $X$ be a set, $F_1, F_2 : X \to C_{b}^{0}(\mathbb{R})$ two maps and $G:\mathbb R\to \mathbb R$. We define the following operations:
\begin{description}
     \item[ Addition $\oplus$] $F_1\oplus F_2 : X \to C_{b}^{0}(\mathbb{R})$ is defined to 
    \[(F_1\oplus F_2)(\mathbf{x})(x):=F_1(\mathbf{x})(x)+F_2(\mathbf{x})(x)\]
    for all $\mathbf{x}\in X$ and all $x\in\mathbb R$, where $+$ denotes the addition on $\mathbb{R}$;
    \item[Multiplication $\bullet$] $G \bullet F_1: X\to C_{b}^{0}(\mathbb{R})$ is defined to be 
    \[(G\bullet F_1 )(\mathbf{x})(x):= G(x) \cdot F_1(\mathbf{x})(x)\] for all $\mathbf{x}\in X$ and all $x\in\mathbb R$, where $\cdot$ denotes the multiplication on $\mathbb{R}$;
    \item[Multiplication $\odot$] $F_1 \odot F_2 : X \to C_{b}^{0}(\mathbb{R})$ is defined to be 
    \[(F_1 \odot F_2)(\mathbf{x})(x):= F_1(\mathbf{x})(x) \cdot F_2(\mathbf{x})(x)\]
    for all $\mathbf{x}\in X$ and all $x\in\mathbb R$, where $\cdot$ denotes the multiplication on $\mathbb{R}$.
\end{description}

\begin{lemma}\label{lm-operation}
 Assume the maps 
 $$S_1, S_2: C_{b}^{0,1}(\mathbb{R},L) \times C_{b,\rho_1}^0(\mathbb{R})  \times C_{b,\rho_2}^0(\mathbb{R})\times \cdot \cdot \cdot \times C_{b,\rho_k}^0(\mathbb{R})  \rightarrow C^0_{b,\rho}(\mathbb{R})$$
 are continuous with respect to the first $k$ variables, where $\rho_i,\rho>0$, $1\le i\le k$, are  constants. Then so are $(G_1 \bullet S_1) \oplus
 (G_2\bullet S_2)$ and $S_1\odot S_2$, where $G_1$, $G_2 \in C_{b}^{0}(\mathbb{R})$.
\end{lemma}
\begin{proof}[Proof]
For any $(\phi,\phi_1,\cdot \cdot \cdot,\phi_{k-1},\phi_{k}), (\tilde{\phi},\tilde{\phi}_1,\cdot \cdot \cdot,\tilde{\phi}_{k-1},\phi_{k}) \in C_{b}^{0,1}(\mathbb{R},L) \times C_{b,\rho_1}^0(\mathbb{R})\times \cdot \cdot \cdot \times C_{b,\rho_k}^0(\mathbb{R}) ,$ we have 
\begin{align*}
&\|\mathcal{G}_1 \cdot S_1(\phi,\phi_1,\cdot \cdot \cdot,\phi_{k-1},\phi_{k})+\mathcal{G}_2 \cdot S_2(\phi,\phi_1,\cdot \cdot \cdot,\phi_{k-1},\phi_{k})\\
&\qquad-\mathcal{G}_1 \cdot S_1(\tilde{\phi},\tilde{\phi}_1,\cdot \cdot \cdot,\tilde{\phi}_{k-1},\phi_{k})-\mathcal{G}_2 \cdot S_2(\tilde{\phi},\tilde{\phi}_1,\cdot \cdot \cdot,\tilde{\phi}_{k-1},\phi_{k})\|\\
&\le \|\mathcal{G}_1\|\cdot\|S_1(\phi,\phi_1,\cdot \cdot \cdot,\phi_{k-1},\phi_{k})-S_1(\tilde{\phi},\tilde{\phi}_1,\cdot \cdot \cdot,\tilde{\phi}_{k-1},\phi_{k})\|\\
&\qquad+\|\mathcal{G}_2\|\cdot \|S_2(\phi,\phi_1,\cdot \cdot \cdot,\phi_{k-1},\phi_{k})-S_2(\tilde{\phi},\tilde{\phi}_1,\cdot \cdot \cdot,\tilde{\phi}_{k-1},\phi_{k})\|,
\end{align*}
which directly implies $(G_1 \bullet S_1) \oplus
 (G_2\bullet S_2)$ is continuous  with respect to the first $k$ variables by assumption,
and 
\begin{align*}
&\|S_1(\phi,\phi_1,\cdot \cdot \cdot,\phi_{k-1},\phi_{k}) \cdot S_2(\phi,\phi_1,\cdot \cdot \cdot,\phi_{k-1},\phi_{k})\\
&\qquad-S_1(\tilde{\phi},\tilde{\phi}_1,\cdot \cdot \cdot,\tilde{\phi}_{k-1},\phi_{k}) \cdot S_2(\tilde{\phi},\tilde{\phi}_1,\cdot \cdot \cdot,\tilde{\phi}_{k-1},\phi_{k})\|\\
&\le \|S_1(\phi,\phi_1,\cdot \cdot \cdot,\phi_{k-1},\phi_{k}) \cdot S_2(\phi,\phi_1,\cdot \cdot \cdot,\phi_{k-1},\phi_{k})\\
&\qquad -S_1(\tilde{\phi},\tilde{\phi}_1,\cdot \cdot \cdot,\tilde{\phi}_{k-1},\phi_{k}) \cdot S_2(\phi,\phi_1,\cdot \cdot \cdot,\phi_{k-1},\phi_{k})\|\\
& \qquad +\|S_1(\tilde{\phi},\tilde{\phi}_1,\cdot \cdot \cdot,\tilde{\phi}_{k-1},\phi_{k}) \cdot S_2(\phi,\phi_1,\cdot \cdot \cdot,\phi_{k-1},\phi_{k})\\
&\qquad-S_1(\tilde{\phi},\tilde{\phi}_1,\cdot \cdot \cdot,\tilde{\phi}_{k-1},\phi_{k}) \cdot S_2(\tilde{\phi},\tilde{\phi}_1,\cdot \cdot \cdot,\tilde{\phi}_{k-1},\phi_{k})\|\\
& \le \|S_2(\phi,\phi_1,\cdot \cdot \cdot,\phi_{k-1},\phi_{k})\| \cdot \|S_1(\phi,\phi_1,\cdot \cdot \cdot,\phi_{k-1},\phi_{k})-S_1(\tilde{\phi},\tilde{\phi}_1,\cdot \cdot \cdot,\tilde{\phi}_{k-1},\phi_{k})\|
\\
&\qquad +
\|S_1(\tilde{\phi},\tilde{\phi}_1,\cdot \cdot \cdot,\tilde{\phi}_{k-1},\phi_{k})\|\|S_2(\phi,\phi_1,\cdot \cdot \cdot,\phi_{k-1},\phi_{k})-S_2(\tilde{\phi},\tilde{\phi}_1,\cdot \cdot \cdot,\tilde{\phi}_{k-1},\phi_{k})\|\\
&\le \rho \|S_1(\phi,\phi_1,\cdot \cdot \cdot,\phi_{k-1},\phi_{k})-S_1(\tilde{\phi},\tilde{\phi}_1,\cdot \cdot \cdot,\tilde{\phi}_{k-1},\phi_{k})\|\\
&\qquad+ \rho\|S_2(\phi,\phi_1,\cdot \cdot \cdot,\phi_{k-1},\phi_{k})-S_2(\tilde{\phi},\tilde{\phi}_1,\cdot \cdot \cdot,\tilde{\phi}_{k-1},\phi_{k})\|,
\end{align*}
which gives rise to that $S_1\odot S_2$ is continuous with respect to the first $k$ variables by assumption. The proof is complete.
\end{proof}

\begin{lemma} \label{coninuous of fg}
Under the assumption of Theorem \ref{th}, for every $j=1,2,\cdots,n,$ the maps $\mathcal{F}_j:  C^{0,1}_b(\mathbb{R};L) \to C^0_b(\mathbb{R})$ defined by \[\mathcal{F}_j(\phi):=(h^{-1})^{(j)}\circ (\phi^{2} -g )\circ f^{-1} \qquad \text{for every } \phi\in C^{0,1}_b(\mathbb{R};L),  \]
where $(h^{-1})^{(j)}$ denotes the $j$-th derivative, the map $\mathcal{G}: C^0_{b,\rho }(\mathbb{R})  \times C^{0,1}_b(\mathbb{R};L) \to  C^0_{b,\rho }(\mathbb{R})$ defined by 
\[\mathcal{G} (\psi,\phi):= \psi \circ \phi \circ f^{-1} \qquad  \]
for every $(\psi,\phi) \in C^0_{b,\rho }(\mathbb{R})\times C^{0,1}_b(\mathbb{R};L)$, and the map $\mathcal{B}:  C^0_{b,\rho }(\mathbb{R}) \to  C^0_{b,\rho }(\mathbb{R})$ defined by
$$ \mathcal{B}(\phi) := \phi \circ f^{-1} \qquad\text{ for every } \phi \in C_{b,\rho}^{0}(\mathbb{R}), $$
are all continuous.

\end{lemma}
\begin{proof}
 First, we prove that $\mathcal{F}_j: C^{0,1}_b(\mathbb{R};L) \to C^0_b(\mathbb{R})$ is continuous for each $j=1,2,\cdots,n$. For any $\phi_0 \in C_{b}^{0,1}(\mathbb{R};L)$, we need to prove that, for any $\epsilon >$ 0, there exists $\delta >0$, such that 
  $$\|(h^{-1})^{(j)} \circ (\phi^{2} \circ f^{-1}-g \circ f^{-1})-(h^{-1})^{(j)} \circ (\phi_0^{2} \circ f^{-1}-g \circ f^{-1})\|<\epsilon$$ whenever $\|\phi-\phi_0\| < \delta$.
  Since the functions $(h^{-1})^{(j)}$ is continuous on $\mathbb{R}$, it is uniformly continuous on the bounded closed interval $I:= [-\|\phi_0\|-\|g\|-1,\|\phi_0\|+\|g\|+1]$. Thus, there exists a $\delta_0  \in (0,1)$ such that
  \begin{equation}\label{unifrom-conti-h}
   |(h^{-1})^{(j)}(x_1)-(h^{-1})^{(j)}(x_2)|< \frac{\epsilon}{2} \text{ whenever } |x_1-x_2|< \delta_0 \text{ and } x_1,x_2 \in I.   
  \end{equation}
  When $\|\phi-\phi_0\| < \frac{\delta_0}{L+1},$ we see that
  \begin{align}\label{pp}
|\phi^{2} \circ f^{-1}(x)-g \circ f^{-1}(x)| &\le \sup_{x \in \mathbb{R}}|\phi^{2} \circ f^{-1}(x)-g \circ f^{-1}(x)|\nonumber
\\
&=
\|\phi^{2} \circ f^{-1}-g \circ f^{-1}\| \nonumber\\
&\le \|\phi\| +\|g\| \le \|\phi-\phi_0\| +\|\phi_0\|+\|g\|
\nonumber\\
& \le \frac{\delta_0}{L+1} +\| \phi_0 \|+\|g\|
\nonumber\\
&\le 1+\| \phi_0 \|+\|g\| \qquad  \text{ for all } x \in \mathbb{R},
  \end{align}
that is, $ \phi^{2} \circ f^{-1}(x)-g \circ f^{-1}(x) \in I$ for all $ x \in \mathbb{R}$; and that
   \begin{align}\label{mm}
  |\phi^{2} \circ f^{-1}(x)- \phi_0^{2} \circ f^{-1}(x)| &\le \|\phi^{2}-\phi_0^{2}\| \le \|\phi^{2}-\phi \circ \phi_0\| +\|\phi \circ \phi_0-\phi_0^{2}\| \nonumber\\ &\le (L+1)\|\phi-\phi_0\| < \delta_0 \qquad  \text{ for all } x \in \mathbb{R}.
  \end{align}
  Consequently,  when $\|\phi-\phi_0\| < \frac{\delta_0}{L+1},$ by \eqref{unifrom-conti-h}, \eqref{pp} and \eqref{mm}, 
  \begin{align*}
  &\|(h^{-1})^{(j)}\circ (\phi^{2} \circ f^{-1}-g \circ f^{-1})-(h^{-1})^{(j)}\circ (\phi_0^{2}\circ f^{-1}-g \circ f^{-1})\| \\
  &= \sup_{x \in \mathbb{R}}|(h^{-1})^{(j)}\circ (\phi^{2} \circ f^{-1}(x)-g \circ f^{-1}(x))-(h^{-1})^{(j)}\circ (\phi_0^{2} \circ f^{-1}(x)-g \circ f^{-1}(x))|\\
  &\le \frac{\epsilon}{2} < \epsilon.
  \end{align*}
  The continuity of $\mathcal{F}_j(\phi), j=1,2,\cdot \cdot \cdot,n,$ is proved.  
  Next, we prove that $\mathcal{G} (\psi,\phi)$ is continuous. For any $(\tilde{\psi},\tilde{\phi}) \in C_{b,\rho}^{0} (\mathbb{R}) \times C_{b}^{0,1} (\mathbb{R};L)$, we need to prove that, for any $\epsilon >$ 0, there exists $\delta >0$, such that
  $$\|\psi \circ \phi \circ f^{-1}-\tilde{\psi} \circ \tilde{\phi} \circ f^{-1}\| < \epsilon \text{ whenever } d((\psi,\phi),(\tilde{\psi},\tilde{\phi})):=\max\{\| \phi-\tilde \phi\|,\| \psi-\tilde \psi\|\}<\delta.$$
Since 
\begin{align}
    &\|\psi \circ \phi \circ f^{-1}-\tilde{\psi} \circ \tilde{\phi} \circ f^{-1}\|\notag\\ 
    &\le \|\psi \circ \phi \circ f^{-1}-\tilde{\psi} \circ \phi \circ f^{-1}\|+\|\tilde{\psi} \circ \phi \circ f^{-1}-\tilde{\psi} \circ \tilde{\phi} \circ f^{-1}\|\notag\\
    &\le \|\psi-\tilde \psi\|+\|\tilde{\psi} \circ \phi \circ f^{-1}-\tilde{\psi} \circ \tilde{\phi} \circ f^{-1}\|. \notag
\end{align}
Thus prove the continuity of $\mathcal{G}$ is complete under $\tilde{\psi} \circ \phi \circ f^{-1}$ is continuous. Next, we just need to prove $\tilde{\psi} \circ \phi \circ f^{-1}$ is continuous. In other words, for every $\epsilon >0$, there exists $\delta >0,$ such that
$$\|\tilde{\psi} \circ \phi \circ f^{-1}-\tilde{\psi} \circ \tilde{\phi} \circ f^{-1}\| < \epsilon \text{ whenever } \|\phi-\tilde \phi\|<\delta.$$
Since the function $\tilde \psi$ is continuous on $\mathbb{R}$, it is uniformly continuous on bounded closed interval $I: =[-\|\tilde{\phi}\|-1,\|\tilde{\phi}\|+1]$. That is to say, for any $\epsilon >0$, there exists $0<\delta_0 <1$, such that
\begin{equation}\label{A}
    |\tilde \psi(x_1)-\tilde \psi(x_2)| < \frac{\epsilon}{2} \text{ whenever } |x_1-x_2|<\delta_0  \text{ and } x_1,x_2 \in I.\\
\end{equation}
In that case, when $\|\phi-\tilde{\phi}\| < \delta_0$,
\begin{equation}\label{AA}
    |\phi \circ f^{-1}(x)| \le \|\phi\| \le \|\phi-\tilde{\phi}\|+\|\tilde{\phi}\| < 1+ \|\tilde{\phi}\|
\qquad \text{ for all } x \in \mathbb{R},
\end{equation}
that is, $\phi \circ f^{-1}(x) \in I$  \qquad \text{ for all } $x \in \mathbb{R};$
and that 
\begin{equation}\label{AAA}
    |\phi \circ f^{-1}(x)-\tilde{\phi} \circ f^{-1}(x)| \le \|\phi-\tilde{\phi}\| < \delta_0 \qquad \text{ for all } x \in \mathbb{R}.
\end{equation}
Thus, when $\|\phi-\tilde{\phi}\| < \delta_0, $ by \eqref{A},\eqref{AA} and \eqref{AAA}, 

$\|\tilde \psi \circ \phi \circ f^{-1}-\tilde \psi \circ \tilde{\phi} \circ f^{-1}\| =\sup_{x \in \mathbb{R}}|\tilde\psi \circ \phi \circ f^{-1}(x)-\tilde\psi \circ \tilde{\phi} \circ f^{-1}(x)| \le \frac{\epsilon}{2} < \epsilon.$
The continuity of $\mathcal{G} (\psi,\phi)$ is proved. 
Similarly, we are able to show that $\mathcal{B}$ is also continuous with respect to $\phi$. The proof is complete.
\end{proof}

\section{Proof of Theorem \ref{th}}\label{proof}

\begin{proof}[Proof of Theorem \ref{th}]
First, we transfer equivalently equation \eqref{o4} to a new equation
\begin{equation}\label{11}
\phi(x)=h^{-1}(\phi^{2}(f^{-1}(x))-g(f^{-1}(x))),~ x\in \mathbb{R}.
\end{equation}
It suffices to show that the functions $h$ and $f$ are bijections on $\mathbb{R}$. 
In fact, by (\ref{o}), if $h'(x)\geq k>1>0$, then $h$  is strictly increasing on $\mathbb{R}$. 
This means that $h$ is injective on $\mathbb{R}$.
For any $x > 0$, by the Lagrange mean value theorem, we have
$h(x) - h(0) = h'(\xi) \cdot x \geq kx$, where $\xi $ is between $ 0 $ and $ x $. As $x \to +\infty $,  clearly $ kx + h(0)$ goes to $+\infty$ since $k>1>0$, which implies that $h(x) \geq kx + h(0)$ also goes to $+\infty$.
For any $ x < 0 $, we see that
$h(x) - h(0) = h'(\xi) \cdot x \leq kx$. Thus, we have $h(x)\to -\infty$ as  $ x \to -\infty$.
By the mean Value Theorem, $h$ is surjective from $\mathbb{R}$ to $\mathbb{R}$. This shows that $h:\mathbb{R}\to \mathbb{R}$ is a bijection. Similarly, if $h'(x)<-1<0$ for all $x\in \mathbb{R}$, we can also show that $h$ is a bijection from $\mathbb{R}$ to $\mathbb{R}$. It is similar to prove that $f:\mathbb{R}\to \mathbb{R}$ is a bijection.

We define a bundle map
\begin{equation}\label{L}
\Gamma:C_{b}^{0,1}(\mathbb{R};L)\times  C_{b,\rho_1}^0(\mathbb{R}) \times  \dots \times C_{b,\rho_n}^0(\mathbb{R})\rightarrow C_{b}^{0,1}(\mathbb{R};L) \times C_{b,\rho_1}^0(\mathbb{R})\times\dots\times C_{b,\rho_n}^0(\mathbb{R})
\end{equation}
by
$$\Gamma(\phi,\phi_{1},\phi_{2},\cdot\cdot \cdot,\phi_{n})=(\Lambda(\phi),\Psi_1(\phi,\phi_1),\Psi_2(\phi,\phi_1,\phi_2),\cdot\cdot \cdot,\Psi_n(\phi,\phi_1,\dots,\phi_n))$$
for $(\phi,\phi_{1},\phi_{2},\dots,\phi_n)\in C_{b}^{0,1}(\mathbb{R};L)\times  C_{b,\rho_1}^0(\mathbb{R}) \times  C_{b,\rho_2}^0(\mathbb{R})\times \dots\times C_{b,\rho_n}^0(\mathbb{R}),$
where $\Lambda:C_{b}^{0,1}(\mathbb{R};L)\rightarrow C_{b}^{0,1}(\mathbb{R};L)$ is defined by
\begin{equation}\label{qq}
\Lambda(\phi):=h^{-1}\circ(\phi^{2}-g)  \circ f^{-1} \qquad \text{~for all ~~}~~\phi \in C_{b}^{0,1}(\mathbb{R};L),
\end{equation}
and with $$\Lambda_k(\phi,\phi',\dots,\phi^{(k)}):=(\Lambda(\phi))^{(k)} \qquad\text{for each } 1 \le k \le n,$$ 
$\Psi_k(\phi,\phi_1,\dots,\phi_k)$ is defined by replacing $\phi',\phi'',\dots,\phi^{(k)}$ with $\phi_1,\phi_2,...,\phi_k$ in \newline $\Lambda_k(\phi,\phi',\dots,\phi^{(k)})$ for all $(\phi,\phi_1,\cdot\cdot \cdot,\phi_k) \in C_{b}^{0,1}(\mathbb{R};L)\times  C_{b,\rho_1}^0(\mathbb{R}) \times\dots\times  C_{b,\rho_k}^0(\mathbb{R}).$ 
Namely, according to Fa\`a di Bruno’s formula given in Lemma \ref{FDB}, for each  $1 \le k \le n$,
$\Psi_k:C_{b}^{0,1}(\mathbb{R};L)\times  C_{b,\rho_1}^0(\mathbb{R})\times\cdot\cdot \cdot\times C_{b,\rho_k}^0(\mathbb{R})  \rightarrow  C_{b,\rho_k}^0(\mathbb{R}) $ is defined by 
\begin{align}
&\Psi_k(\phi,\phi_1,\cdot\cdot \cdot,\phi_k):=\Lambda_k(\phi,\phi_1,\cdot\cdot \cdot,\phi_k)     =(h^{-1})'\circ (\phi^{2}  -g)\circ f^{-1} \cdot \notag\\
& \qquad \bigg( (\phi_k \circ \phi \circ f^{-1} )\cdot (\phi_1 \circ f^{-1})^k + (\phi_1\circ \phi \circ f^{-1}) \cdot (\phi_k \circ f^{-1})\bigg) \cdot ((f^{-1})')^k\notag
\\
& \qquad +H_k(\phi,\phi_1,\phi_2,\cdots,\phi_{k-1}), \label{psik}
\end{align}
where
\[H_k(\phi,\phi_1,\phi_2,\cdots,\phi_{k-1})\]
is finite operations of addition, subtraction, multiplication and composition among $h^{-1}, g, f^{-1},$ and their derivatives as well as $\phi, \phi_1,\cdots,\phi_{k-1}$. Namely,
\begin{equation}\label{H_k}
   \begin{aligned}
 &H_k (\phi,\phi_1,\phi_2,\cdots,\phi_{k-1}) := \\
 &(h^{-1})'\circ (\phi^{2} -g)\circ f^{-1} \cdot \bigg\{ \phi_1 \circ \phi  \circ f^{-1} \cdot \bigg( \sum_{h_1+2h_2+\dots+kh_k=k \atop h_1 \neq k} \frac{k!}{h_1!h_2!\dots h_k!} \phi_{h_1+h_2+\dots +h_k} \circ \\
 &f^{-1} \cdot\prod_{m=1}^k\big(\frac{( f^{-1} )^{(m)}}{m!}\big)^{h_m} \bigg)  -(g\circ f^{-1})^{(k)}  + \\
 &\sum_{t_1+2t_2+\dots+(k-1)t_{k-1}=k-1 }\frac{(k-1)!}{t_1!t_2!\dots t_{k-1}!(1!)^{t_1}(2!)^{t_2}\dots{((k-1)!)}^{t_{k-1}}} \phi_{t_1+t_2+\dots +t_{k-1}}\circ \\
 &\phi \circ f^{-1} \cdot
  \prod_{j=1}^{k-1}\bigg(   \sum_{r_1+2r_2+\dots+jr_j=j }\frac{j!}{r_1!r_2!\dots r_j!} \phi_{r_1+r_2+\dots +r_j} \circ f^{-1} \cdot \prod_{p=1}^j\big(\frac{ ( f^{-1} )^{(p)}}{p!}\big)^{r_p}       \bigg)^{t_j} \bigg\}   +  \\
&\sum_{l_1+2l_2+\dots+(k-1)l_{k-1}=k-1 }\frac{(k-1)!}{l_1!l_2!\dots l_{k-1}!(1!)^{l_1}(2!)^{l_2}\dots((k-1)!)^{l_{k-1}}} \cdot (h^{-1})^{(l_1+l_2+\dots +l_{k-1})}\circ\\
&(\phi^{2} - g)\circ f^{-1} \cdot 
  \prod_{i=1}^{k-1}\bigg\{   \sum_{s_1+2s_2+\dots+is_i=i }\frac{i!}{s_1!s_2!\dots s_i!(1!)^{s_1}(2!)^{s_2}\dots (i!)^{s_i}} \phi_{s_1+s_2+\dots +s_i} \circ \\
  &\phi \circ f^{-1} \cdot  \prod_{u=1}^i\bigg(  \sum_{v_1+2v_2+\dots+uv_u=u }\frac{u!}{v_1!v_2!\dots v_u!} \phi_{v_1+v_2+\dots +v_u} \circ f^{-1} \cdot \prod_{w=1}^u(\frac{(f^{-1})^{(w)}}{w!})^{v_w}    \bigg)^{s_u} \\
  &-(g\circ f^{-1})^{(i)}\bigg\}^{l_i}   
   \end{aligned}
\end{equation}
for all $(\phi,\phi_1,\cdot\cdot \cdot,\phi_k) \in C_{b}^{0,1}(\mathbb{R};L)\times  C_{b,\rho_1}^0(\mathbb{R})\times\cdot\cdot \cdot\times C_{b,\rho_k}^0(\mathbb{R}).$

In (\ref{L}),  $L$ is chosen to satisfy
\begin{equation}\label{3}
    \begin{cases}
        \frac{m\alpha-\sqrt{{m}^{2}
\alpha^{2}-4\beta}}{2} \le &L \le \frac{m\alpha
+\sqrt{{m}^{2}\alpha^{2}-4\beta}}{2},\\
&L < m-1,
    \end{cases}
\end{equation}
$\rho_1$ is chosen to satisfy
\begin{equation}\label{44}
    \frac{m\alpha-\sqrt{m^2\alpha^2-4\beta}}{2}\le \rho_1 \le \frac{m\alpha+\sqrt{m^2\alpha^2-4\beta}}{2} ,
 \rho_1 < \frac{m\alpha}{2}
\end{equation}
and
\begin{equation}\label{55}
    \begin{cases}
\rho_1^2+\rho_1-m\alpha^2&<0,\\
\rho_1^3+\rho_1-m\alpha^3&<0,\\
\cdot \cdot \cdot \\
\rho_1^n+\rho_1-m\alpha^n&<0,
\end{cases}
\end{equation}
as well as for each $2\le k\le n$,  $\rho_k$ is chosen to satisfy
\begin{align}\label{8}
\rho_k\ge \frac{m\alpha^k\lambda_k}{m\alpha^k-{\rho_1}^k-\rho_1}.
\end{align}
It can be seen that such an $L$ exists as long as $\beta$ is sufficiently small. Similarly, there exists a $\rho_1$ that satisfies \eqref{44} and \eqref{55}  provided that $\beta$ is sufficiently small. Thus, for each $2\le k\le n$, a $\rho_k$ satisfying \eqref{8} exists.

We claim that the maps $\Lambda$ and $\Psi_k, 1\le k\le n,$ are well defined by ($\ref{qq}$) and ($\ref{psik}$) respectively if $L$ and $\rho_k(k=1,2,\cdot \cdot \cdot,n)$ satisfy conditions (\ref{3})--(\ref{8}).
In fact, 
it is shown in \cite{tx} that $\Lambda $ maps $C_{b}^{0,1}(\mathbb{R};L)$ into itself if (\ref{3}) holds, which means that $\Lambda$ is well defined. Moreover, proved in \cite{tx},  $\Psi_1$ maps $C_{b}^{0,1}(\mathbb{R};L) \times C_{b,\rho_1}^0(\mathbb{R})  $ into $C_{b,\rho_1}^0(\mathbb{R})$, that is, $\Psi_1$ is well defined.
Evidently,
  $\Psi_k(\phi,\phi_1,\phi_2,\cdot\cdot \cdot,\phi_k)$ is continuous on $\mathbb{R}$ for each $2\le k\le n$ and every $(\phi,\phi_1,\phi_2,\cdot\cdot \cdot,\phi_k) \in C_{b}^{0,1}(\mathbb{R};L) \times C_{b,\rho_1}^{0}(\mathbb{R}) \times C_{b,\rho_2}^{0}(\mathbb{R})\times\cdot\cdot \cdot\times C_{b,\rho_k}^{0}(\mathbb{R})$.
  
  Moreover, note that, shown in \eqref{H_k}, $H_k$ contains $\phi, \phi_1,\cdots,$ and $\phi_{k-1}$, not containing $\phi_k$. Since $\|\phi_i\| \le \rho_i, i=1,2,...,k-1,$ and the derivatives $(h^{-1})^{(i)}, (f^{-1})^{(i)}$ and $g^{(i)}$, $i=1,\cdots,k$, are all bounded under the hypotheses, all operations in \eqref{H_k} preserve boundedness, there exists a constant $\lambda_k<\infty$ which 
   depends only on $\rho_1,\cdots,\rho_{k-1},  \|(h^{-1})^{(i)}\|$ and $\|(f^{-1})^{(i)}\|, i=1,2,...,k,$ such that
   \[ \|H_k(\phi,\phi_1,\phi_2,\cdots,\phi_{k-1})\| \le \lambda_k. \]
Consequently, it follows by (\ref{o}) and \eqref{8}  that
\begin{align}
        &\|\Psi_k(\phi,\phi_1,\phi_2,\cdot\cdot \cdot,\phi_k)\| \le \|(h^{-1})'\circ (\phi^{2}  -g)\circ f^{-1}\| \cdot  \notag\\
        &\bigg\|(\phi^{(k)} \circ \phi \circ f^{-1}) \cdot  (\phi' \circ f^{-1})^k + (\phi'\circ \phi \circ f^{-1}) \cdot (\phi^{(k)} \circ f^{-1}) \bigg\| \cdot \|((f^{-1})')^k\| \notag \\
        &+ \|H_k(\phi,\phi_1,\phi_2,\cdots,\phi_{k-1})\| \notag\\
        &\le \frac{\rho_k\rho_1^k+\rho_k\rho_1}{m\alpha^k}+\lambda_k \notag\\
        &\le \rho_k, \qquad \text{ (because $\rho_k$ is chosen to satisfy \eqref{8}) } \notag
\end{align}
which shows that $\Psi_k$ is well defined for each $2\le k\le n$ from $C_{b}^{0,1}(\mathbb{R};L)\times  C_{b,\rho_1}^0(\mathbb{R}) \times  C_{b,\rho_2}^0(\mathbb{R})\times\cdot\cdot \cdot\times C_{b,\rho_k}^0(\mathbb{R}) $ to $ C_{b,\rho_k}^0(\mathbb{R}).$

We further have the following three assertions, whose proofs will be given after finishing the proof.
\begin{description}
    \item[(A1)] \label{a1} $\Lambda$ is contraction on $C_{b}^{0,1}(\mathbb{R};L).$ 
    \item[(A2)] $\Psi_k(k=1,2,\cdot\cdot \cdot,n)$ is a uniform contraction with respect to the first $k$ variables.
    \item[(A3)]  $\Psi_k(\cdot, \phi_{k\ast}):  C_{b}^{0,1}(\mathbb{R};L)\times  C_{b,\rho_1}^0(\mathbb{R}) \times  C_{b,\rho_2}^0(\mathbb{R})\times\cdot\cdot \cdot\times C_{b,\rho_{k-1}}^0(\mathbb{R}) \rightarrow C_{b,\rho_{k}}^0(\mathbb{R})(k=1,2,\cdot\cdot \cdot,n)$ is continuous, where $\phi_{k\ast}$ denotes the component at the 
$(k+1)$-th position in the fixed point $(\phi_{\ast}, \phi_{1\ast}, \phi_{2\ast}, \dots, \phi_{n\ast})$ of $\Gamma$ obtained by Lemma \ref{fct}.
\end{description}
Thus, by Lemma \ref{fct}, this fixed point of $\Gamma$ is globally attractive. In other words, for every $(\phi,\phi_{1},\phi_{2},\cdot\cdot \cdot,\phi_{n}) \in C_{b}^{0,1}(\mathbb{R};L) \times C_{b,\rho_1}^0(\mathbb{R}) \times C_{b,\rho_2}^0(\mathbb{R})\times\cdot\cdot \cdot\times C_{b,\rho_n}^0(\mathbb{R}),$ 
\begin{equation}\label{w}
    \text{ $ \Gamma^{m}(\phi,\phi_{1},\phi_{2},\cdot\cdot \cdot,\phi_{n})$ converges to $(\phi_{\ast},\phi_{1\ast},\phi_{2\ast},\cdot\cdot \cdot,\phi_{n\ast})$ as $m$ $\rightarrow +\infty.$ } 
\end{equation}
Choose arbitrarily $\phi_{0} \in C_{b}^{0,1}(\mathbb{R};L),\phi_{10} \in C_{b,\rho_1}^0(\mathbb{R}),\cdot\cdot \cdot,\phi_{n0} \in C_{b,\rho_n}^0(\mathbb{R})$ such that $\phi_{10}=\phi_{0}'$, $\phi_{20}=\phi_{0}'',\cdot\cdot \cdot,\phi_{n0}=\phi_{0}^{(n)}.$
 Let 
\begin{equation}\label{wwww}
    (\phi_{m},\phi_{1m},\phi_{2m},\cdot\cdot \cdot,\phi_{nm}):=\Gamma^{m}(\phi_{0},\phi_{10},\phi_{20},\cdot\cdot \cdot,\phi_{n0}).
\end{equation}
Because of the definitions of $\Lambda,\Psi_1, \Psi_2,\cdot\cdot \cdot,\Psi_n$, we have that 
\begin{equation}\label{www}
    \phi_{m}'=\phi_{1m},\phi_{m}''=\phi_{2m},\cdot\cdot \cdot,\phi_{m}^{(n)}=\phi_{nm}  \text{  ~for all } m \geq 0. 
\end{equation}
By \eqref{w}, \eqref{wwww} and \eqref{www}, it follows that $\phi_{\ast}'=\phi_{1\ast}, \phi_{\ast}''=\phi_{2\ast},\cdot\cdot \cdot,\phi_{\ast}^{(n)}=\phi_{n\ast}$, implying that $\phi_{\ast}$ is of class $C^{n}$ and its first to $n$-th derivatives are bounded by $\rho_1, \rho_2,\cdot\cdot \cdot,\rho_n$ respectively.
Note that $\phi_{\ast}$ is a fixed point of $\Lambda$, which is a solution of equation  \eqref{o4}. The proof is complete.
\end{proof}
Now we are in a position to prove assertions {\bf (A1)}-{\bf (A3)}.
In fact,  it is shown in \cite[page 6]{tx} that assertion {\bf (A1)} is proved. 
\begin{proof}[Proof of assertion {\bf (A2)}]
By \eqref{H_k}, for any $k=1,2,\dots,n,$ there is no $\phi_k$ in $H_k$. Thus, by \eqref{o} and \eqref{psik}, for any $\phi \in C_{b}^{0,1}(\mathbb{R};L), \phi_1 \in C_{b,\rho_1}^0(\mathbb{R}),\cdot\cdot \cdot,\phi_{k-1} \in C_{b,\rho_{k-1}}^0(\mathbb{R})$, and any $\phi_k, \tilde{\phi}_{k} \in C_{b,\rho_k}^0(\mathbb{R}) $, we have that
\begin{align*}
&\|\Psi_k(\phi,\phi_1,\phi_2,\cdot\cdot\cdot,\phi_k)-\Psi_k(\phi,\phi_1,\phi_2,\cdot\cdot\cdot,\tilde{\phi}_k)\|\\
&= \sup_{x \in \mathbb{R}}|\Psi_k(\phi,\phi_1,\phi_2,\cdot\cdot\cdot,\phi_k)(x)-\Psi_k(\phi,\phi_1,\phi_2,\cdot\cdot\cdot,\tilde{\phi}_k)(x)|\\
&=\sup_{x \in \mathbb{R}}\Big| (h^{-1})'\circ(\phi^2-g)\circ f^{-1}(x) \cdot\bigg(\big(\phi_k\circ\phi\circ f^{-1}(x)-\tilde{\phi_k}\circ\phi\circ f^{-1}(x)\big) \cdot
\\
&\qquad (\phi_1 \circ f^{-1}(x))^k+(\phi_1\circ \phi\circ f^{-1}(x)) \cdot \big((\phi_k \circ f^{-1}(x))-\tilde{\phi_k}\circ f^{-1}(x)\big)\bigg)\Big|\cdot
\\
&\qquad |((f^{-1})'(x))^k|\\
&\le \frac{1}{m\alpha^k}\bigg(\sup_{x \in \mathbb{R}}\big|\phi_k\circ\phi\circ f^{-1}(x)-\tilde{\phi_k}\circ\phi\circ f^{-1}(x) \big| \cdot \rho_1^k\\
&\qquad+\sup_{x \in \mathbb{R}}\big|\phi_k\circ f^{-1}(x)-\tilde{\phi_k}\circ f^{-1}(x) \big| \cdot \rho_1 \bigg)\\
&\le \frac{\rho_1^k+\rho_1}{m\alpha^k}\|\phi_k-\tilde{\phi_k}\|.
 \end{align*}
Therefore, $\Psi_k(k=1,2,\cdot \cdot \cdot,n)$ is a uniform contraction with respect to the first $k$ variables by \eqref{44} and (\ref{55}) . The proof  of assertion {\bf(A2)} is complete.
\end{proof}

\begin{proof}[Proof of assertion {\bf(A3)}]
Let
\begin{align*}
      \mathcal{F}_j(\phi) &:=(h^{-1})^{(j)}\circ (\phi^{2} -g) \circ f^{-1}, \text{ where } \phi \in C_b^{0,1}(\mathbb R;L),\\
         \mathcal{G}_j(\phi_j,\phi) &:=\phi_{j} \circ \phi \circ f^{-1},  \text{  where } \phi \in C_b^{0,1}(\mathbb R;L) \text{ and } \phi_j \in C_{b,\rho_j}^{0}(\mathbb{R},L),\\
           \mathcal{B}_j(\phi_j) &:= \phi_{j} \circ f^{-1}, \text{ where } \phi_j \in C_{b,\rho_j}^{0}(\mathbb{R},L).
\end{align*}
Due to Fa\`a di Bruno's Formula shown in Lemma \ref{FDB} and the definition of $\Psi_k$, to prove the continuity of $\Psi_k(\cdot, \phi_{k\ast})(k=1,2,\cdot \cdot \cdot,n)$ only needs to show that $\mathcal{F}_i(\phi), 1\le i\le k, \mathcal{G}_j(\phi_j,\phi)$ and $  \mathcal{B}_j(\phi_j), j=1,2,\cdots,k-1, $ are all continuous by Lemma \ref{lm-operation}. By Lemma \ref{coninuous of fg}, the proof  of assertion {\bf(A3)} follows.
\end{proof}

\section*{Acknowledgements}
This work is supported by the Science and Technology Research Program of Chongqing Municipal Education Commission (Grant No. KJQN202300540).

\end{document}